\documentclass[10pt, a4paper, reqno]{amsart}

\usepackage{amsmath, amssymb, amsfonts, enumerate}
\usepackage[english]{babel}
\usepackage[utf8]{inputenc}
\usepackage{graphicx}
\usepackage{amsfonts}
\usepackage{mathrsfs}
\usepackage{stmaryrd}
\usepackage{nicefrac}
\usepackage{hyperref}
\usepackage{graphicx}
\usepackage{epic}

\usepackage[usenames]{color}

\theoremstyle{definition}

\newtheorem*{defin*}{Definition}
\theoremstyle{plain}
\newtheorem{corollary}{Corollary}
\newtheorem{prop}{Proposition}
\newtheorem{remark}{Remark}

\newtheorem{lemma}{Lemma}
\newtheorem*{teo*}{Theorem}
\newtheorem{teo}{Theorem}
\newtheorem{conjecture}{Conjecture}

\newcommand{\Z}{\mathbb{Z}}
\newcommand{\Q}{\mathbb{Q}}
\newcommand{\R}{\mathbb{R}}
\newcommand{\C}{\mathbb{C}}
\newcommand{\T}{\mathbb{T}}

\newcommand{\scalar}[2]{\langle #1, #2 \rangle}

\makeatletter
\def\imod#1{\allowbreak\mkern10mu({\operator@font mod}\,\,#1)}
\makeatother

\newcommand{\mm}[2]{\llbracket #1 \rrbracket_{#2}}

\newcommand{\defeq}{\stackrel{\mathrm{def}}{=}}
\newcommand{\Oh}{\mathcal{O}}

\allowdisplaybreaks

\begin{document}

\title[An Algebraic Kronecker's Theorem and Linear Recurrences]{An equivalent of Kronecker's
Theorem for powers of an Algebraic Number and Structure of Linear Recurrences of fixed length}
\author[N. Dubbini]{Nevio Dubbini}
\address{Interdepartmental Research Center ``E.Piaggio'' - Via Diotisalvi 2 - 56126 Pisa}
\email{nevio.dubbini@for.unipi.it}

\author[M. Monge]{Maurizio Monge}
\address{Scuola Normale Superiore di Pisa - Piazza dei Cavalieri, 7 - 56126 Pisa}
\email{maurizio.monge@sns.it}
\date{\today}

\begin{abstract}
  After defining a notion of $\epsilon$-density, we provide for any real algebraic number $\alpha$
  an estimate of the smallest $\epsilon$ such that for each $m>1$ the set of vectors of the form
  $(t,t\alpha,\dots,t\alpha^{m-1})$ for $t\in\R$ is $\epsilon$-dense modulo $1$, in terms of the
  multiplicative Mahler measure $M(A(x))$ of the minimal integral polynomial $A(x)$ of $\alpha$,
  and independently of $m$. In particular, we show that if $\alpha$ has degree $d$ it is possible
  to take $\epsilon = 2^{[d/2]}/M(A(x))$.

  On the other hand using asymptotic estimates for Toeplitz determinants we show that
    for sufficiently large $m$ we cannot have $\epsilon$-density if $\epsilon$ is a fixed
    number strictly smaller than $1/M(A(x))$. As a byproduct of the proof we obtain a result
  of independent interest about the structure of the $\Z$-module of integral linear recurrences of
  fixed length determined by a non-monic polynomial.
\end{abstract}

\maketitle

\section{Introduction}

The classical Kronecker's Theorem in diophantine approximation, in one of its different versions,
says that if $\theta_1,\dots,\theta_m\in{}\R$ are linearly independent over $\Q$ then
$(t\theta_1,\dots,t\theta_m)$ is dense modulo $1$. This means, denoting by
\[
 \pi_m : \R^m \rightarrow \R^m/\Z^m = \T^m
\]
the canonical projection, that
$\left\{\pi_m(t\theta_1,\dots,t\theta_m),\ t \in \R\right\}$ is dense in $\T^m$.

In this paper we investigate the problem of giving an estimate about how much this results
  fails when the $\theta_i$ have special form. In particular, we give an estimate, uniformly
  in $m$, when the $(\theta_i)$ are powers of an algebraic number $\alpha$, or more
  generally when they statisfy a general linear recurrence relation with characteristic polynomial
$A(x)$ of degree $d$, a condition which is automatically satisfied if $\theta_i=\alpha^{i-1}$
and $A(x)$ is the minimal integral polynomial of $\alpha$. We specify that a sequence
$\theta_1,\theta_2,\dots$ is a recurrence sequence determined by the polynomial
$a_dx^d+\dots+a_1x+a_0$ when $\sum_{i=0}^d\theta_{i+j}a_i=0$ for each $j>1$.

Before stating our main result, we give a definition of $\epsilon$-density and recall the notion of
Mahler measure.

\subsection{\texorpdfstring{$\epsilon$}{epsilon}-density} For $\epsilon >0$ and a positive integer
$m$, let $I_\epsilon=\left[-\nicefrac{\epsilon}{2},\nicefrac{\epsilon}{2}\right]^m$ be
  the cube with edge lenght $\epsilon$ centered in the origin, and let
$C_\epsilon=\pi_m\left(I_\epsilon\right)\subseteq\T^m$ be its projection on the torus.

\begin{defin*}
  A subset $S\subseteq \T^m$ is \emph{$\epsilon$-dense} if $S+C_{\bar{\epsilon}}=\T^m$ for
  each $\bar{\epsilon}>\epsilon$, or equivalently if $S+C_\epsilon$ is dense. A subset
  $T\in\R^m$ is $\epsilon$-dense if $T+I_{\bar{\epsilon}}=\R^m$ whenever $\bar\epsilon >
  \epsilon$, or equivalently if $T+I_\epsilon$ is dense.
\end{defin*}

Alternatively, it is possible to consider on $\T^m$ the distance $d_\infty(x,y)$ defined as the
infimum of the $\infty$-distance $|\tilde x - \tilde y|_\infty$ over all
representatives $\tilde{x},\tilde{y}\in\R^m$ of $x,y$. Since for each $\rho>0$ the
$\rho/2$-neighborhood of $S$ with respect to $d_\infty(\cdot,\cdot)$ is precisely $S+C_\rho$, we
have that $S$ will be $\epsilon$-dense if and only if its $\bar{\epsilon}/2$-neighborhood is
the whole $\T^m$ for each $\bar{\epsilon} >\epsilon$ (or, equivalently, if its
$\epsilon/2$-neighborhood is dense).

\subsection{Mahler measure}
{Let} $A(x)=\sum_{i=0}^da_ix^i=a_d\cdot\prod_{i=1}^d\big(x-\alpha_i\big)$ be a
  polynomial with complex coefficients such that $a_0\neq{}0$. The \emph{Mahler measure} of $A(x)$
is defined as
\[
  M(A) = M(A(x)) = \left|a_d\right| \cdot \prod_{i=1}^d \max\big\{1,|\alpha_i|\big\}.
\]
It is a notion of complexity that for the minimal polynomial of a rational number $p/q$ with
$(p,q)=1$ reduces to $\max(|p|,|q|)$. Furthermore, each coefficient $a_i$ can be written as the sum
of $\binom{d}{i}$ products of some elements among $a_d,\alpha_1,\dots,\alpha_d$, each being
taken at most once, and we have that $a_i \leq \binom{d}{i} \cdot M(A)$, for $i=0,\dots,d$.

\subsection{Netwon polygon}\label{newtsec}

For a prime $p$ let $v_p(\cdot)$ be the $p$-adic valuation, and let us recall that the Newton
polygon of an integral primitive polynomial $A(x)$ is the boundary of the polygon in $\R^2$ obtained
as the upper convex envelop of the points with coordinates
\[
  (0,v_p(a_0)), (1,v_p(a_1)),\dots, (d,v_p(a_d)).
\]
Suppose that the polygon has $r$ edges with different slopes which connect the consecutive pairs of
points $(w_k, v_p(a_{w_k}))$ for $0 = w_0<w_1<\dots<w_r = d$, let $\ell_k = w_k - w_{k-1}$ be the
length of the horizontal projection of the $k$-th edge, and let
$\sigma_k=(v_p(a_{w_k})-v_p(a_{w_{k-1}}))/\ell_k$ be its slope. We have that $\sigma_i < \sigma_j$
whenever $1\leq i < j \leq d$.

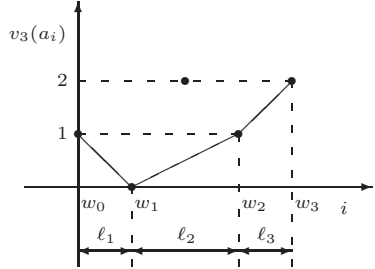
\begin{center}
\[
\setlength{\unitlength}{20pt}
\begin{picture}(7,4.5)(0,-5.5)
 \dashline[-15]{0.15}(0,-2)(4,-2)
 \dashline[-15]{0.15}(0,-3)(3,-3)
 \dashline[-15]{0.15}(4,-2)(4,-5.5)
 \dashline[-15]{0.15}(3,-3)(3,-5.5)
 \dashline[-15]{0.15}(1,-4)(1,-5.5)

 \drawline(0,-3)(1,-4)
 \drawline(1,-4)(3,-3)
 \drawline(3,-3)(4,-2)
 \put(0,-3){\circle*{0.15}}
 \put(1,-4){\circle*{0.15}}
 \put(3,-3){\circle*{0.15}}
 \put(4,-2){\circle*{0.15}}
 \put(2,-2){\circle*{0.15}}

 \put(-1,-4){\vector(1,0){6.5}}
 \put(0,-5.5){\vector(0,1){5}}
 
 \put(-0.4,-3.1){\scriptsize{$1$}}
 \put(-0.4,-2.1){\scriptsize{$2$}}
 \put(-1.3,-1.2){\scriptsize{$v_3(a_i)$}}

 \put(4.9,-4.5){\scriptsize{$i$}}
 \put(0.05,-4.4){\scriptsize{$w_0$}}
 \put(1.05,-4.4){\scriptsize{$w_1$}}
 \put(3.05,-4.4){\scriptsize{$w_2$}}
 \put(4.05,-4.4){\scriptsize{$w_3$}}

 \put(0.35,-5){\scriptsize{$\ell_1$}}
 \put(1.85,-5){\scriptsize{$\ell_2$}}
 \put(3.35,-5){\scriptsize{$\ell_3$}}
 \put(0.5,-5.2){\vector(1,0){0.5}}
 \put(0.5,-5.2){\vector(-1,0){0.5}}
 \put(2,-5.2){\vector(1,0){1}}
 \put(2,-5.2){\vector(-1,0){1}}
 \put(3.5,-5.2){\vector(1,0){0.5}}
 \put(3.5,-5.2){\vector(-1,0){0.5}}

\end{picture} \]
{\bf Newton polygon of $A(x) = 9x^4 - 3x^3 -9x^2 - 2x + 3$ for $p=3$.}
\end{center}

We will consider a splitting field $K$ of $A(x)$ over the field of $p$-adic numbers $\Q_p$, and
we extend the $p$-adic valuation and absolute value to $K$. As usual, we call
$\mathcal{O}_K$ the integral closure of $\Z_p$ in $K$, which is the local ring formed by the
elements having non-negative valuation, with maximal ideal $\mathfrak{m}$ generated by the
uniformizer $\pi$.  As it is well known, the polynomial $A(x)$ has precisely $\ell_k$ roots with
$p$-adic valuation equal to $-\sigma_k$, for $1\leq k \leq r$.

\subsection{Main results} We now state our main results:
\begin{teo}
\label{teo1}
Let $m, d>0$ be positive integers, and let $\theta_1,\dots,\theta_m$ be real numbers such that
$\theta_1,\dots,\theta_d$ are linearly independent over $\Q$, and the remaining
$\theta_{d+1},\dots,\theta_m$ are inductively defined by the linear recurrence relation
  induced by a primitive integral polynomial $A(x)$ of degree $d$ and with nonzero constant
coefficient. Then
\[
  S_\theta = S_{\theta,m} = \left\{\pi_m(t\theta_1,t\theta_2,\dots,t\theta_m),\ t \in \R\right\}
   \subseteq \T^m
\]
is $\epsilon$-dense for each $\epsilon$ at least
\[
  \min\left\{\frac{1}{M(A(x/2))}, \frac{2^d}{M(A(2x))}\right\}.
\]
\end{teo}

In other terms, if $\epsilon$ is as requested, then for arbitrary real numbers
$x_1,\dots,x_m$ it is possible to find a real number $t$ and integers $p_1,\dots,p_m$ such that
\[
  x_i \leq t\theta_i-p_i \leq x_i+\epsilon,\qquad\text{for }1\leq i \leq m.
\]

The $(\theta_i)$ can clearly be taken to be the powers $(\alpha^{i-1})$ of a real
algebraic number $\alpha$, and $A(x)$ the minimal integral polynomial of $\alpha$. But for
each primitive integral polynomial $A(x)$ the allowed $(\theta_i)$ are any
`sufficiently generic' recurrence sequence determined by $A(x)$, provided that $\theta_1,\dots,\theta_d$
are linearly independent over $\Q$.

We remark for convenience that
\[
  \min\left\{\frac{1}{M(A(x/2))}, \frac{2^d}{M(A(2x))}\right\} \leq \frac{2^{[d/2]}}{M(A(x))},
\] 
obtaining an estimate in terms of the Mahler measure of $A(x)$, this inequality will be
  also proved later.

Fixing the recurrence sequence $\theta_1,\theta_2,\dots$ and varying $m$, the best
$\epsilon_m$ such that $S_{\theta,m}$ is $\epsilon_m$-dense is initially equal to $0$ by
  Kronecker's theorem, because $S_{\theta,m}$ is dense for $m\leq{}d$, and then increases with $m$,
being $S_{\theta,m}$ a projection of $S_{\theta,m'}$ for $m < m'$. Theorem \ref{teo1}
gives an upper bound for the sequence $\epsilon_m$, a lower bound for the limit shall
be given in Theorem \ref{teo2} below.

Some computational evidence actually makes us propose the following
\begin{conjecture}
In Theorem \ref{teo1} we have $\epsilon$-density for each $\epsilon$ at least $1/M(A(x))$.
\end{conjecture}

The conjecture is supported by the apparent connection of the above problem with algebraic
dynamics (see \cite{everest1999heights}): as it is well known, the toral automorphism
determined by the companion matrix of the polynomial $A(x)$ has topological entropy equal to
$M(A(x))$, a quantity controlling the growth of orbits (which are recurrence sequences connected to
$A(x)$). The $\epsilon$-density considered here is rather about how well the orbits in the toral
dynamical system can approximate a generic sequence modulo $1$, but there may still be some
connection between the two questions.

This conjectured optimal density coefficient is not very far from what is obtained in Theorem
\ref{teo1}. Such a result would be optimal because of the following lower bound for the
$\epsilon$-density, for large $m$:

\begin{teo} \label{teo2}
  Let $\theta_1,\theta_2,\dots$ and $A(x)$ be as in Theorem \ref{teo1}. Then for each
  $\epsilon<1/M(A(x))$ the set $S_{\theta,m}$ is not $\epsilon$-dense for sufficiently large $m$.
\end{teo}

Note that if the existence of some $S_{\theta,m}$ which are not $\epsilon$-dense
  for some $\epsilon<1$ could be ruled out, not considering the instances where $A(x)$
is a cyclotomic polynomial, then a positive solution of Lehmer problem
\cite[Chap.~1]{everest1999heights} would follow. This is because the existence of a sequence
of polynomials with Mahler measure approaching $1$ from above would be automatically ruled out.

During the proof of this Theorem a result about the structure of the module of linear recurrences of
fixed length is obtained. In particular, fixing a prime $p$, denoting as
$\Lambda_m^{(p)} \subset \Z_p^m$ the module of linear recurrences of length $m$ in $\Z_p$ determined
by $A(x)$, and keeping the notation of \S\ref{newtsec} for the Newton polygon of $A(x)$, we
have that

\begin{teo}
\label{teo3}
For each prime $p$ there exists a unique basis of $\Lambda_m^{(p)}$ such that the $d\times m$ matrix
$M = (M_{i,j})$ having the basis vectors as rows satisfies
\begin{enumerate}
 \item the left $d\times d$ submatrix $(M_{i,j})_{\substack{1\leq i\leq d\\1\leq j\leq d}}$
(resp. the submatrix $(M_{i,j})_{\substack{1\leq i\leq d\\m-d+1\leq j\leq m}}$) is a
block upper (resp. lower) triangular matrix, with blocks $B_1,\dots,B_r$ (resp. $C_1,\dots,C_r$) on
the diagonal; the number of blocks $r$ is equal to the number of edges of the Newton polygon of
$A(x)$, and $B_k$ and $C_k$ are square matrices of size equal to the
lenght $\ell_k$ of the $k$-th side of the polygon for each $1\leq k \leq r$;
\item taking $1\leq s\leq r$ to be the smallest integer such that $\sigma_s \geq 0$
  (or, alternatively, such that $\sigma_s > 0$) then $B_1,\dots,B_{s-1},C_s,\dots,C_r$ are identity
  matrices, for $1\leq k \leq s-1$ the matrix $C_k$ has determinant with valuation
  $-\sigma_k\ell_k(m-d)$, and for $s\leq k \leq r$ the matrix $B_k$ has determinant with valuation
  $\sigma_k\ell_k(m-d)$;
\item moving right from the element $M_{j,j}=1$ contained in the diagonal of the $k$-th
  identity block $B_k$ for $1 \leq k \leq s-1$ (resp. moving left from an $M_{j+m-d,j}=1$
  contained in the diagonal of the identity block $C_k$, for $s \leq k \leq r$) by $t$ steps we find
  elements with $p$-adic valuation at least $-t\sigma_k$ (resp. at least $t\sigma_k$).
\end{enumerate}
\end{teo}

In other words, the matrix $M$ can be taken with the following form:
\[
\newcommand*{\ddd}[1]{ \makebox[0.1in][c]{$\ddots$} }
\newcommand*{\dd}[1]{ \makebox[0.11in][c]{$\dots$} }
\newcommand*{\zero}[1]{ \makebox[0.11in][c]{$0$} }
\newcommand*{\one}[1]{ \multicolumn{1}{|c|}{\makebox[0.1in][c]{$1$}} }
\newcommand*{\CC}[1]{ \multicolumn{1}{|c|}{\makebox[0.1in][c]{$C_{#1}$}} }
\newcommand*{\CCs}[1]{ \multicolumn{1}{|c|}{\makebox[0.1in][c]{$\scriptstyle{C_{#1}}$}} }
\newcommand*{\BB}[1]{ \multicolumn{1}{|c|}{\makebox[0.1in][c]{$B_{#1}$}} }
\newcommand*{\BBs}[1]{ \multicolumn{1}{|c|}{\makebox[0.1in][c]{$\scriptstyle{B_{#1}}$}} }
\newcommand\TT{\rule{0pt}{2.3ex}}
\left(
\begin{array}{cccccccc|cc|cccccccc}
 \cline{1-1}\cline{11-11} \one & & &  \dd &  &  &  &  & &   & \dd &   &    &
	\CC{0} & & & & & & & \TT \\
 \cline{1-2}\cline{11-12} & \one &  &  &  \dd &  &  &  &  &   &  & \dd &   &
	& \CC{1} & & & & & \zero & \TT \\
 \cline{2-2}\cline{12-12} &  & \ddd & &  &  &    &   &   &   &   &
	& & \ddd & \TT \\
 \cline{4-4}\cline{14-14} &  &  & \one &  & & \dd  & &     & &  \dd  &  &
	&  & \dd & & & \CCs{s-1} &  \TT \\
 \cline{4-5}\cline{14-15} &  &  & & \BB{s} &  & \dd  &  &   &  &  \dd  &  &  &
	&  \dd & & & \one &  \TT \\
 \cline{5-5}\cline{15-15} & &  &  &  &  \ddd &  &  &  &         &  &
	& & & & & \ddd &  \TT \\
 \cline{7-7}\cline{17-17} & \zero & & & &  &  & \BBs{r-1} &  &  \dd &  &  &
	& & & & \dd & & & \one &  \TT \\
 \cline{7-8}\cline{17-18} & &  &  &  &  &  & \BB{r}  &  &   \dd &  &
	& & & & & \dd & &  & \one &  \TT \\
 \cline{8-8}\cline{18-18} 
\end{array}\right)
\]
and the behaviour of the $p$-adic valuation on the rows is controlled by the slope of the segments
of the Newton polygon of $A(x)$.

For instance, if $A(x) = 9x^4 - 3x^3 -9x^2 - 2x + 3$ and we take $p=3$, then for recurrence length
$m=10$ the matrix $M$ is
\[
\newcommand*{\BB}[1]{ \multicolumn{1}{|c|}{#1} }
\newcommand*{\LL}[1]{ \multicolumn{1}{|c}{#1} }
\newcommand*{\RR}[1]{ \multicolumn{1}{c|}{#1} }
\newcommand\TT{\rule{0pt}{2.6ex}}
\newcommand\SP{\rule[-1ex]{0pt}{0pt}}
\left(
\begin{array}{cccccccccc}
\cline{1-1}\cline{7-7}
  \BB{1} &\frac{480}{887}& \frac{4203}{16853}& \frac{3861}{33706}& \frac{2511}{33706}& 
\frac{243}{16853}&  \BB{\frac{729}{33706}} &  0 &  0 &  0 \SP\TT \\
\cline{1-3}\cline{7-9} 0& \LL{\frac{-7722}{887}} & \RR{\frac{-5049}{16853}}
&\frac{-3339}{33706}&\frac{-76419}{ 33706}&\frac{33378}{16853} & \frac{-51543}{33706} & \LL{1} & 
\RR{0} &  0 \SP\TT \\
 0 & \LL{\frac{729}{887}}&\RR{\frac{-44658}{16853}}&\frac{42822}{16853} & \frac{-27306}{16853}
  &\frac{19179}{16853}& \frac{3489}{16853} &  \LL{0}  &  \RR{1} & 0 \SP\TT \\
\cline{2-4}\cline{8-10}  0 \SP\TT &  0 &   0 & \BB{\frac{729}{1078}}&   \frac{243}{1078}&   
\frac{405}{539}&   \frac{675}{1078}&   \frac{423}{539}&   48/49 & \BB{1} \\
\cline{4-4}\cline{10-10}
\end{array}
\right).
\]
All denominators are prime with $3$, the $3$-adic valuations are
\[
\left(
\begin{array}{cccccccccc}
  0&  1&  2&  3&  4&  5&  6&  \infty&  \infty&  \infty \\
  \infty&  3&  3&  2&  2&  1&  3&  0&  \infty&  \infty \\
  \infty&  6&  3&  3&  2&  2&  1&  \infty&  0&  \infty \\
  \infty&  \infty&  \infty&  6&  5&  4&  3&  2&  1&  0 
\end{array}
\right),
\]
and these valuations increase going right in the first row, slowly increase going left in
the second and third rows, and increase going left on the last row.

\subsection{Acknowledgements} We wish to thank Roberto Dvornicich for offering kind support and
advice during our work, and for providing us the motivation to improve the second part of the
paper. We would also like to thank Albrecht B\"ottcher and Samuele Mongodi for their help in
looking for references, and Lisa Beck for reading the manuscript.

\section{Preliminaries}
For a complex polynomial $A(x) = a_d x^d + \dots + a_1 x + a_0$
with $a_d,a_0 \neq 0$ let us define, for each $\ell \geq 1$, the
rectangular $\ell \times (\ell+d)$ matrix
\[ \mm{A}{\ell} \defeq
\begin{pmatrix}
  a_0 & a_1 & \dots & a_d & \\
      & a_0 & a_1   & \dots & a_d  \\
      &      & a_0 & a_1  & \dots& a_d  \\
      &      &     & \ddots &    &  \ddots  & \ddots \\
      &      &     &        & a_0 & a_1 & \dots & a_d  \\
\end{pmatrix}.
\]
$\mm{A}{\ell}$ is also the multiplication matrix for sections of power series by the polynomial
$A(x)$, in the sense that if
\[
F(x) = \sum_{i=0}^\infty f_i x^i, \qquad G(x) = \sum_{i=0}^\infty g_i x^i
\]
are complex power series such that $G(x) = A(x) \cdot F(x)$, then
\[
 \left( g_{d+\ell-1}, \dots, g_{d+1}, g_d \right)^\top =
\mm{A}{\ell} \cdot
  \left( f_{d+\ell-1}, \dots, f_{d+1}, f_d, \dots, f_0  \right)^\top,
\]
indeed the range of coefficients $0,1,\dots,d+\ell-1$ of $F(x)$ uniquely determines the
coefficients $d,d+1,\dots,d+\ell-1$ of $G(x)$.

Moreover, if $A(x)$ factors as
\[
  A(x) = B(x) \cdot C(x),
 \qquad B(x) = \sum_{i=0}^s b_i x^i,
 \qquad C(x) = \sum_{i=0}^t c_i x^i,
\]
then we have
\[
  \mm{A}{\ell} = \mm{B}{\ell} \cdot \mm{C}{\ell+s} = \mm{C}{\ell} \cdot \mm{B}{\ell+t}
\]
as it is easy to verify, observing that multiplying a power series by $A(x)$ and discarding the
lowest $d$ coefficients gives the same result as multiplying first by $C(x)$ and discarding the
lowest $t$ coefficients, and then by $B(x)$ and discarding the lowest $s$ coefficients (and
vice-versa).

For each polynomial $A(x)$ and positive integer $m$ we will also need the lower triangular $m \times
m$ matrix defined as
\[ \{A\}_m \defeq
\begin{pmatrix}
  a_d &    &      &    & \\
  a_{d-1} & a_d   &   &     &      \\
 \vdots  &  \vdots & \ddots &    &    &  \\
 a_0  & a_1 &  \dots  & a_d &    &     &      \\
      & a_0 & a_1   & \dots &  a_d   &  \\
      &     & \ddots &  \ddots &        & \ddots  &   \\
      &     &        &  a_0 &  a_1  &    & a_d  \\
\end{pmatrix}.
\]
Note that $\mm{A}{\ell}$ is embedded in $\{A\}_{\ell+d}$ as the last $\ell$ rows. Similarly to the
previous case, we have that if $A(x) = B(x) \cdot C(x)$, then $\{A\}_m = \{B\}_m \cdot \{C\}_m$. In
particular, if $C(x)$ splits in linear factors as
\[ C(x) = \prod_{i=1}^t\left(x-\gamma_i\right), \]
we have that
\[ \{C\}_m = \prod_{i=1}^t \{x-\gamma_i\}_m. \]

\section{Proof of Theorem \ref{teo1}}

We will first show the inequality implicitly stated in Theorem \ref{teo1}:
\begin{prop}
Let $A(x)$ be a polynomial, then
\[
  \min\left\{\frac{1}{M(A(x/2))}, \frac{2^d}{M(A(2x))}\right\} \leq \frac{2^{[d/2]}}{M(A(x))}.
\]
\end{prop}
\begin{proof}
We have
\begin{align*}
  M(A(x/2)) &= \frac{\left|a_d\right|}{2^d} \cdot \prod_{j=1}^d
\max\big\{1,|2\alpha_j|\big\}\\
  &= \left|a_d\right| \cdot \prod_{j=1}^d \max\left\{\frac{1}{2},|\alpha_j|\right\},
\end{align*}
while
\begin{align*}
  2^{-d}M(A(2x)) &= \left|a_d\right| \cdot \prod_{j=1}^d
\max\left\{1,\left|\frac{\alpha_j}{2}\right|\right\}\\
  &= \left|a_0\right| \cdot \prod_{j=1}^d
\max\left\{\left|\alpha_j^{-1}\right|,\frac{1}{2}\right\} = M(\tilde{A}(x/2)),
\end{align*}
where we called $\tilde{A}(x)$ the conjugate polynomial $\sum_{i=0}^d a_i x^{d-i}$.

We have $M(A(x/2)) \geq 2^{-[d/2]}M(A(x))$ if $A(x)$ has at least $d/2$ roots $\geq 1$ in absolute
value, while $2^{-d}M(A(2x)) \geq 2^{-[d/2]}M(A(x))$ if $A(x)$ has at least $d/2$ roots $\leq 1$ in
absolute value. Since at least one of these conditions must hold the proposition is proven.
\end{proof}

We now prove Theorem \ref{teo1}.

\begin{proof}[Proof of Theorem \ref{teo1}]
From Kronecker's Approximation Theorem \cite[pg. 53-54]{cassels1957introduction} the
set $S_{\theta,m}$ is dense in the closed subgroup of the torus of the elements satisfying
the same integral relations as the $\theta_1,\dots,\theta_m$.

Let $A(x) = \sum_{i=0}^d a_dx^d = a_d\prod_{i=1}^d(x-\alpha_i)$ be the primitive integral polynomial
defining the recurrence sequence $\theta_1,\theta_2,\dots$ . If $k_0,\dots,k_{m-1}$ is an integral
linear relation satisfied by the $\theta_i$ (that is $\sum_{i=1}^m k_{i-1}\theta_i = 0$), and
$K(x) = \sum_{i=0}^{m-1} k_i x^i$ is the corresponding polynomial, we can consider the remainder
$R(x)=\sum_{i\geq 0} r_ix^i$ of the division of $K(x)$ by $A(x)$. It is a polynomial of degree
$\leq{}d-1$ and its coefficients $r_0,r_1,\dots$ still define a linear relation of the
$\theta_i$. But $\theta_1,\dots,\theta_d$ are linearly independent over $\Q$, so $R(x)$ must be
zero.

Consequently $K(x)$ must be a multiple of $A(x)$ in $\Q[x]$, and thus in $\Z[x]$ by Gauss
lemma, being $A(x)$ primitive. This shows that for $\ell=m-d$ the rows of the matrix $\mm{A}{\ell}$
generate the $\Z$-module of the integral relations of $\theta_1,\dots,\theta_m$.

The closure of $S_{\theta,m}$ is thus the projection on $\T^m$ of the hyperplane
\[
  P = P(A) = \Big\{ v \in \R^m\ :\ \mm{A}{\ell} \cdot v = 0 \Big\},
\]
the kernel of the linear application $\mm{A}{\ell} : \R^m \rightarrow \R^\ell$.

Its inverse image under the projection $\pi_m$ is $P + \Z^m$, and since we just showed that the
sublattice of $\Z^m$ generated by the rows of the matrix $\mm{A}{\ell}$ is the saturated sublattice
of the integral linear relations of $\theta_1,\dots,\theta_m$, we have that $P+\Z^m$ is also
equal to
\[
  Q = Q(A) = \Big\{ v \in \R^m\ :\ \mm{A}{\ell} \cdot v \in \Z^\ell \Big\}.
\]
Indeed, the $\Z$-linear map $\mm{A}{\ell} : \Z^m \rightarrow \Z^\ell$ must be surjective
(this is true a fortiori being the above lattice saturated, but it can also be seen directly
considering the $\ell\times \ell$ minors modulo $p$ for each prime $p$ \cite[Lemma 2,
Chap. 1]{cassels1971ign}, which have gcd $1$ because $A(x)$ is primitive). Now, $P+\Z^m\subseteq{}Q$
clearly, and for each vector $w \in Q$ there exist a vector $z \in \Z^m$ such that
$\mm{A}{\ell}\cdot{}w=\mm{A}{\ell}\cdot{}z$, and consequently the difference $v = w-z$ is in $P$,
and we have that $w = v+z \in P+\Z^m$.

We are now reduced to prove that $Q(A)$ is $\epsilon$-dense in $\R^m$, for each $\epsilon$
satisfying the inequality stated in the theorem. Applying the involution
\[
  (x_1,x_2,\dots,x_{m-1},x_m) \mapsto (x_m,x_{m-1},\dots,x_2,x_1)
\]
to $Q(A)$ we obtain a set which clearly has the same $\epsilon$-density properties as $Q(A)$, and
which is $Q(\tilde{A})$, where $\tilde{A}(x)$ is the conjugated polynomial
$\tilde{A}(x)=a_0x^d+\dots+a_{d-1}x+a_d$. Consequently, we can just prove that $Q(A)$ is
$\epsilon$-dense for
\[
 \epsilon \geq 1/M(A(x/2)) = \left(\left|a_d\right| \cdot \prod_{i=1}^d
\max\left\{\frac{1}{2},|\alpha_i|\right\}\right)^{-1},
\]
and the $\epsilon$-density for $\epsilon \geq 2^{d}/M(A(2x)) = 1/M(\tilde{A}(x/2))$ will follow from
the same estimate applied to $Q(\tilde{A})$.

To obtain this estimate, note first that $Q = Q(A)$ will be $\epsilon$-dense if and only if
the image of $I_\epsilon=I_{\epsilon,m}=[-\nicefrac{\epsilon}{2},\nicefrac{\epsilon}{2}]^m$
via the map $\pi_\ell \circ \mm{A}{\ell}$ is the whole $\T^\ell$. Indeed, $Q$ is
$\epsilon$-dense if and only if an arbitrary vector $v\in\R^m$ is contained in $Q+I_\epsilon$, and
\begin{gather*}
 \Leftrightarrow\qquad\left(v + I_\epsilon\right) \cap Q \neq \emptyset\quad\forall v\in\R^m\\
 \Leftrightarrow\qquad\left(w + \mm{A}{\ell}\cdot I_\epsilon\right) \cap \Z^\ell \neq
\emptyset\quad\forall w\in\R^\ell
\end{gather*}
applying the matrix $\mm{A}{\ell}$ to the expression, and denoting with
$\mm{A}{\ell}\cdot{}I_\epsilon$ the image of $I_\epsilon$ under the map $\mm{A}{\ell}$.  This
passage must be justified because the matrix $\mm{A}{\ell}$ does not have rank $m$, but since $Q$
contains \emph{all} the vectors that are mapped to $\Z^\ell$ the first intersection will be
non-empty whenever the second one is (the other direction is trivial).

We will now factor $A(x)$ as a product of polynomials $B(x),C(x) \in \R[x]$, with a consequent
factorization of $\mm{A}{\ell}$, and prove two different estimates for each of the two factors
$B(x)$ and $C(x)$. We will select \emph{a posteriori} the factorization providing the best
compound estimate.

So, let $B(x)=\sum_{i=0}^s b_ix_i$ and $C(x)=\sum_{i=0}^t c_ix_i$, with
$A(x)=B(x)\cdot{}C(x)$, and recall the induced matrix factorization $\mm{A}{\ell} = \mm{B}{\ell}
\cdot \mm{C}{\ell+s}$. To prove that
\[
  \pi_\ell \circ \mm{A}{\ell} = \pi_\ell \circ \mm{B}{\ell} \circ \mm{C}{\ell+s}
\]
is surjective from $I_{\epsilon,m}$ to $\T^\ell$, we can just prove that $\pi_\ell \circ
\mm{B}{\ell}$ is surjective from $I_{\delta,\ell+s}$ for some $\delta$, and that
$\mm{C}{\ell+s}\cdot I_{\epsilon,m}$ contains $I_{\delta,\ell+s}$. We anticipate that $C(x)$ will be
chosen monic and with all roots $<1$ in absolute value.

\subsection{Estimate for \texorpdfstring{$B(x)$}{B(x)}}
We can take $\delta = 1/b_0$. Let us show that the image the of cube $I_{\delta,\ell+s}$ under
$\mm{B}{\ell}$ assumes each value modulo $\Z^\ell$, starting with an arbitrary
$v=(v_1,\dots,v_{\ell}) \in \R^\ell$ and building inductively a vector $w = (w_1,\dots,w_{\ell+s})$
in $I_{\delta,\ell+s}$ such that $\mm{B}{\ell}\cdot w - v \in \Z^\ell$. Suppose that
$w\in{}I_{\delta,\ell+s}$ is such that the components with index $>i$ of $\mm{B}{\ell}\cdot w-v$ are
in $\Z$, and observe that while $w_{i}$ varies in the interval
$[-\nicefrac{\delta}{2},\nicefrac{\delta}{2}]$ the $i$-th component of $\mm{B}{\ell}\cdot w$ varies
in an interval large $b_0\delta = 1$, while all components with index $>i$ of $\mm{B}{\ell}\cdot w$
stay fixed. Consequently we can select $w_{i}$ in the interval
$[-\nicefrac{\delta}{2},\nicefrac{\delta}{2}]$ to ensure that all components with index $\geq i$ of
$\mm{B}{\ell}\cdot w - v$ are in $\Z$. Repeating this procedure we construct a $w$
with the required properties and our claim follows.

\subsection{Estimate for \texorpdfstring{$C(x)$}{C(x)}} If $ C(x) = \sum_{i=0}^t c_ix^i =
\prod_{i=1}^t \left(x-\gamma_i\right),$ we prove that it is possible to take
\[
 \epsilon = \delta \cdot \prod_{i=1}^t \frac{1}{1-|\gamma_i|}
\]
as follows. Rather than working with $\mm{C}{\ell+s}$ we work with the nonsingular square matrix
$\{C\}_m$, if we prove that the image under $\{C\}_m$ of $I_{\epsilon,m}$ contains $I_{\delta,m}$
our claim will follow, since the image under $\mm{C}{m}$ is a just projection on the last $\ell+s$
coordinates of the image under $\{C\}_m$.

Now, rather than proving that the image under $\{C\}_m$ of $I_{\epsilon,m}$ contains $I_{\delta,m}$,
it's easier proving that the image under $\{C\}_m^{-1}$ of $I_{\delta,m}$ is contained in
$I_{\epsilon,m}$. In particular, $\{C\}_m$ factors as
\[ \{C\}_m = \prod_{i=1}^t \{x-\gamma_i\}_m, \]
and the inverse of a matrix of the form $\{x-\gamma\}_m$, for $\gamma \in \{\gamma_i\}_{1\leq
i\leq t}$, is easily
computed as
\[
\begin{pmatrix}
 1 &  &    \\
 -\gamma  & 1 &    \\
   & -\gamma  & 1 &    \\
   &   & -\gamma   & \ddots &    \\
   &   &   &   \ddots      & 1   & \\
   &   &   &        & -\gamma   & 1  \\
\end{pmatrix}^{-1} =
\begin{pmatrix}
 1         &          &           &             &       &       \\
 \gamma    &   1      &           &             &       &       \\
 \gamma^2  &  \gamma  &   1       &             &       &       \\
 \gamma^3  & \gamma^2 & \gamma    & \ddots      &       &       \\
 \vdots    &          & \vdots    & \ddots      &    1  &       \\
 \gamma^{m-1} & \gamma^{m-2} & \dots & \ \gamma^2\  & \ \ \gamma\ \  & \ \ 1\ \  \\
\end{pmatrix}.
\]
This shows that if a (possibly complex) vector $v=(v_1,\dots,v_m)$ has all components with absolute
value $\leq \rho$ for some real number $\rho>0$, the vector $w$ obtained applying the matrix
$\{x-\gamma\}_m^{-1}$ will have components of the form
\[ w_{r+1} =  \sum_{i=0}^r \gamma^{i} v_{r+1-i} \]
for some $0\leq r < m$, and their absolute value can be estimated as
\[
  \left|\sum_{i=0}^r \gamma^i v_{r+1-i} \right| 
    \leq \sum_{i=0}^r \left|\gamma^i\right|\cdot \left| v_{r+1-i} \right|
    \leq \sum_{i=0}^r \left|\gamma^i\right|\cdot \rho \leq \frac{1}{1-|\gamma|}\rho.
\]
Since $\rho$ is arbitrary we obtain, applying iteratively $\{x-\gamma_i\}_m^{-1}$ for $i=1,\dots,t$,
that the set of complex vectors with all components $<\delta$ in absolute value is mapped by
$\{C\}_m^{-1}$ to complex vectors whose components have absolute value at most
$\epsilon$. Consequently $I_{\delta,m}$ is mapped into $I_{\epsilon,m}$, being $\{C\}_m^{-1}$ a
matrix with real entries.

\subsection{Conclusion} Let $A(x) = B(x)C(x)$ be a real factoriziation of $A(x)$, with
\[
  B(x) = \sum_{i=0}^s b_i x^i = a_d \cdot \prod_{i=1}^s \left(x-\beta_i\right),
  \qquad C(x) = \sum_{i=0}^t c_i x^i = \prod_{i=1}^t \left(x-\gamma_i\right),
\]
and with $C(x)$ monic with all roots $<1$ in absolute value. The above estimate shows that $Q$ is
$\epsilon$-dense for each $\epsilon$ at least
\[ \frac{1}{|a_d|} \cdot \prod_{i=0}^s \frac{1}{|\beta_i|} \cdot \prod_{i=1}^t
\frac{1}{1-|\gamma_i|},\]
and consequently for $\epsilon$ at least
\[
  \frac{1}{|a_d|} \cdot \prod_{i=0}^d \min\left\{\frac{1}{|\alpha_i|},
   \frac{1}{1-|\alpha_i|}\right\} 
  = \left(\left|a_d\right| \cdot \prod_{i=1}^d \max\Big\{|\alpha_i|,
   1-|\alpha_i|\Big\}\right)^{-1},
\]
because we can take as $\gamma_i$ precisely the $\alpha_i$ with absolute value $\leq 1/2$,
and as $\beta_i$ the remaining roots of $A(x)$ (note that $C(x)$ will have real coefficients). Since
this last expression is clearly not greater than $1/M(A(x/2))$, the proof is complete.
\end{proof}

\begin{remark}
  While this Theorem gives an $\epsilon$ providing $\epsilon$-density for $S_m$ which is good for
  each $m$, and by Theorem \ref{teo2} this cannot be smaller than $1/M(A)$, a consideration on the
  dependence on $m$ of the best possible constant $\epsilon_m$ should be added. Discarding the
  $m\leq{}d$, for which $S_m$ is dense and hence $\epsilon_m = 0$, for $m=d+1$ the matrix
  $\mm{A}{d+1}$ is $(a_0, a_1, \dots, a_d)$ and since we must have
  $\mm{A}{d+1}\cdot{}I_\epsilon+\Z=\R$, $\epsilon$ should be at least
\[ 
\frac{1}{\sum_{i=0}^d|a_i|} \geq \frac{1}{\sum_{i=0}^d\binom{d}{i} M(A)} = \frac{1}{2^d \cdot M(A)}.
\]
Consequently the first non-trivial example already requires a constant of the order of $1/M(A)$, up
to a constant depending only on $d$.
\end{remark}

\section{Integral linear recurrences of fixed length}

In this section we prove Theorem \ref{teo3}, it will be proved in full strength for its
independent interest, even if only a small corollary is required to prove Theorem \ref{teo2}. All
conclusions obtained in $\Z_p$ can be lifted to $\Z$ with arbitrary approximation with
respect to the $p$-adic absolute value.

Let $A(x) = \sum_{i=0}^d a_i x^i$ be a primitive integral polynomial of degree $d$ with
$a_0\neq{}0$, and for $m>d$ let $\Lambda_m$ be the $\Z$-module of the integral vectors in $\Z^m$
making a recurrence sequence determined by $A(x)$. It is also the module of the integral
vectors in the kernel of the matrix $\mm{A}{m-d}$, so it has rank $d$.  The $\Z_p$-module
$\Lambda_m\otimes_\Z\Z_p$ over the ring of $p$-adic integers $\Z_p$ is clearly equal to the
set of vectors in $\Z_p^m$ annihilated by $\mm{A}{m-d}$ and will be denoted by
$\Lambda_m^{(p)}$.

\begin{proof}
  Let $N = (N_{i,j})$ be the rational $d\times m$ matrix obtained putting $N_{i,j} = \delta_{ij}$
  for $1\leq i,j \leq d$, and inductively defining the remaining elements in each row to form a
  linear recurrence determined by $A(x)$. We prove that the square matrix
  $N_\xi=(N_{i,\xi_j})_{1\leq i,j\leq d}$ is non-singular for $\xi=(1,2,\dots,w,m-d+w+1,\dots,m)$,
  for all $w = w_0,\dots,w_r$ which are the ordinate of a vertex of the Newton polygon of $A(x)$.

If $A(x)$ has distinct roots $\alpha_1,\dots,\alpha_d$ the matrix $N$ is given by
$V^{-1} \cdot L$, where
\[
  V = (\alpha_i^{j-1})_{1\leq i,j \leq d},\qquad L = (\alpha_i^{j-1})_{\substack{1\leq i \leq
d\\1\leq j \leq m}}.
\]
To obtain a formula for the determinant of $N_\xi$ valid for general $\alpha_i$, let us work over
$\C$ and suppose for a moment that $\frac{1}{a_d}A(x) = \prod_{i=1}^d (x-\alpha_i)$ where the
$\alpha_i$ are algebraically independent over $\C$. The determinant of
$N_\xi=(N_{i,\xi_j})_{1\leq{}i,j\leq d}$ is equal to the determinant of $L_\xi =
(L_{i,\xi_j})_{1\leq i,j\leq d}$ divided by $\det V$, and this turns out to be the Schur function
$s_\lambda$ (see \cite{macdonald1995}) associated to the partition $\lambda$ defined as
$\lambda_{d-1+i} = \xi_i-i+1$ for $1\leq i \leq d$, evaluated in
$\alpha_1,\dots,\alpha_d$. Applying the definition of $N$ by linear recurrence, note
that the entries of $N$ are polynomial functions in the elementary symmetric functions of the
$\alpha_i$, and hence polynomials in the $\alpha_i$, and the determinant of each submatrix is
a polynomial function of the $\alpha_i$ as well. For $\alpha_i$ varying outside
of the closed algebraic set defined by $\prod_{i>j}(\alpha_i-\alpha_j) = 0$, the determinant of
$N_\xi$ is equal to the polynomial $s_\lambda(\alpha_1,\dots,\alpha_d)$, and consequently such
expression will hold for each value of the $\alpha_i$.

If $\xi$ is defined as above, $\lambda$ has precisely $d-w$ parts all equal to $m-d$, and its
conjugate partition $\lambda'$ is formed by $m-d$ parts equal to $d-w$. Recall now Jacobi-Trudi's
identity
\[
  s_\lambda = \det (e_{\lambda_i'-i+j})_{1\leq i,j \leq k},
\]
where the $e_i$'s are the elementary symmetric functions, which holds for each $k$ at least as big
as the number of parts of $\lambda'$. We are reduced to prove that a matrix of the form
$(e_{d-w-i+j})_{1\leq i,j \leq m-d}$ is non-singular. After evaluation in the roots we have
$e_i=(-1)^ia_{d-i}/a_d$, and flipping the sign of rows and columns of even index we can consider the
determinant of
\[
  U = (a_{w+i-j})_{1\leq i,j \leq m-d}
\]
up to possibly a sign, and discarding a factor $a_d^{-m+d}$. Note that all entries on the diagonal
are equal to $a_w$.

If $w=w_0=0$ (resp. if $w=w_r=d$) then the matrix is lower (resp. upper) triangular with $a_w\neq0$
on the diagonal, and consequently non-singular. Suppose $w = w_k$ for some $1<k<r$, let $\rho \in K$
be an element with $p$-adic valuation equal to $\sigma_k$, and let $R$ be the diagonal matrix with
$1,\rho,\rho^2,\dots,\rho^{m-d-1}$ on the diagonal. The matrix
\[
  \frac{1}{a_w} R^{-1} \cdot U \cdot R = \left(\frac{a_{w+i-j}\rho^{j-i}}{a_w}\right)_{1\leq i,j
\leq m-d}
\]
has all entries in $\mathcal{O}_K$, and is upper unitriangular when reduced modulo $\mathfrak{m}$
because it has all $1$ on the diagonal, and starting from $a_w$ the $p$-adic valuation increases at
a rate bigger than $\sigma_k$.

Consequently $N_\xi$ is non-singular, and we can consider the matrix
$Q=N_\xi^{-1}\cdot{}N$. We obtain that the matrix $M$ must be unique: indeed, assume by
  contradiction $M'$ to satisfy the same properties, and suppose that the $i$-th rows
  differ. If $(v_j)_{1\leq j \leq m}$ is the difference vector of those rows we must have
$v_j = 0$ unless $w_k<j\leq{}m-d+w_k$ for some $k$, and taking $w=w_k$ we have that $(v_j)$
cannot be a linear combination of the rows of $Q$. But $Q$ has rank $d$, and the existence of
$(v_j)$ would imply that the module of linear recurrences $\Lambda_m \otimes_\Z \Q$ has rank $>d$,
which is absurd.

Suppose now that $w=w_k$ with $1 \leq k < s$, so that the slope $\sigma_k$ is $\leq 0$ by definition
of $s$. We show that all entries in the rows with indices $w_{k-1}+1,\dots,w_k-1,w_k$ are integral,
and that the valuation of $Q_{i,j}$ for $w_{k-1} < i \leq w_k$ and $j \geq i$ is at least
$-\sigma_k(j-i)$.

In fact, let $\rho$ have valuation equal to $\sigma_k$ and consider the polynomial
\[
   B(x) = \sum_{i=0}^d \frac{\rho^{w-i} a_i}{a_w} x^i = \sum_{i=0}^d b_i x^i \ \ \in
\mathcal{O}_K[x].
\]
Note that the Newton polygon of $B(x)$ is obtained by the Newton polygon of $A(x)$ by subtracting
the linear affine function $f(x) = \sigma_k(x - w) + v_p(a_w))$, and the $k$-th side of the Newton
polygon is moved to lay on the horizontal axis. In particular, the coefficients $b_{w_{k-1}}$
and $b_{w_k}$ of $B(x)$ are $\not\equiv 0\imod{\mathfrak{m}}$, but $b_i \equiv 0\imod{\mathfrak{m}}$
for $i<w_{k-1}$ or $i > w_k$.

On the other hand, for each $w_{k-1} < i \leq w_k$ the vector $(\rho^{j-i}Q_{i,j})_{1\leq j \leq m}$
is a linear recurrence determined by $B(x)$, and we claim that all entries are in
$\mathcal{O}_K$. Indeed, suppose this is not the case, and multiply its entries by the
smallest power of the uniformizer $\pi$ required to make all entries in $\mathcal{O}_K$. Some entry
will be in $\mathcal{O}_K \setminus \mathfrak{m}$, but the first $\ell_k = w_k$ entries will be in
$\mathfrak{m}$. When reduced modulo $\mathfrak{m}$, the subvector $(\rho^{j-i}Q_{i,j})_{w_{k-1} < j
  \leq m-d+w_k}$ is a recurrence determined by the polynomial
\[
  C(x) = \sum_{i=0}^{\ell_k} \overline{b_{w_{k-1}+i}}x^i = \sum_{i=0}^{\ell_k} c_i x^i\ \ \in
\left(\mathcal{O}_K/\mathfrak{m}\right)[x].
\]
This recurrence of order $\ell_k$ in $\mathcal{O}_K/\mathfrak{m}$ is supposed to have non-zero
entries, while the first $\ell_k$ entries are zero. This is absurd, and the claim on the
  integrality of the $w_{k-1}+1,\dots,w_k-1,w_k$ is proved.

The matrix $(\rho^{j-i}Q_{i,j})_{\substack{w_{k-1} < i \leq
w_k\\m-d+w_{k-1} < j \leq m-d+w_k}}$ is invertible, because modulo $\mathfrak{m}$ it is obtained
as the $(m-d)$-th power of the companion matrix
\[
  \left(
  \begin{array}{ccccc}
  0 & & & & -c_0/c_{\ell_k} \\
  1 & 0 & & & -c_1/c_{\ell_k} \\
  & 1 & 0 & & -c_2/c_{\ell_k} \\
  & & \ddots & \ddots & \vdots \\
  & &  & 1 & -c_{\ell_k-1}/c_{\ell_k} \\
  \end{array}
  \right)
\]
of $C(x)$, which is invertible. Consequently $(Q_{i,j})_{\substack{w_{k-1} < i \leq
w_k\\m-d+w_{k-1} < j \leq m-d+w_k}}$, which is obtained by it conjugating by the diagonal matrix
with diagonal $1,\rho,\dots,\rho^{\ell_k-1}$ and multiplicating by a factor $\rho^{-(m-d)}$, has the
same valuation as $\rho^{-\ell_k(m-d)}$, i.e. equal to $-\sigma_k\ell_k(m-d)$.

In this way we have built the rows from $w_{k-1}+1$ to $w_k$ of $M$, and proved that $C_k =
(Q_{i,j})_{\substack{w_{k-1} < i \leq w_k\\m-d+w_{k-1} < j \leq m-d+w_k}}$ has determinant with
valuation $-\sigma_k\ell_k(m-d)$, while $B_k = (Q_{i,j})_{w_{k-1} < i,j \leq w_k}$ is the
identity. For $s \leq k \leq r$ we can clearly proceed in a symmetrical way, taking $C_k$ equal to
the identity and proceeding on the left to $B_k$, which will have determinant equal to
$\sigma_k\ell_k(m-d)$.

The matrix we have built selecting at each step the rows from $w_{k-1}+1$ to $w_k$ of $Q$ clearly
satisfies all requirements for the matrix $M$. Furthermore all rows are linearly
independent, so the module they generate over $\Z_p$ has rank $d$.

To prove that the rows of $M$ generate all of $\Lambda_m^{(p)}$ observe that they generate a
$\Z_p$-module contained in $\Lambda_m^{(p)}$, and suppose the generated module to be strictly
contained. A basis of $\Lambda_m^{(p)}$ can be obtained by left multiplication by a
matrix $B$ with determinant $\in \Q_p\setminus \Z_p$. However the matrix
$M_\xi=(M_{i,\xi_j})_{1\leq{}i,j\leq{}d}$ for $\xi = (1,\dots,w_s,m-d+w_s+1,\dots,m)$ has
determinant $1$, and $B \cdot M_\xi$ (and consequently $B \cdot M$) would not have coefficients in
$\Z_p$, which is absurd.
\end{proof}

Let us now turn to the $\Z$-module $\Lambda_m$ again. If $E$ is a finitely generated $\Z$-module and
$F\subseteq E$ a submodule with index $n$, we have that the index of $F\otimes_\Z\Z_p$ in
$E\otimes_\Z\Z_p$ is precisely the biggest power of $p$ dividing $n$.

Let $\Theta_m \subseteq \Z^d \times \Q^{m-d}$ be the $\Z$-module of the linear recurrences
determined by $A(x)$ such that the first $d$ coordinates are in $\Z$, and similarly let
$\Theta^{(p)}_m = \Theta_m \otimes_\Z \Z_p$. The matrix $N$ of the proof of the Theorem is clearly a
$\Z$-basis of $\Theta_m$.

For each prime $p$ the matrix $M$ is equal to $M_\xi \cdot N$ where
$M_\xi=(M_{i,\xi_j})_{1\leq{}i,j\leq d}$ and $\xi=(1,2,\dots,d)$, and as it is possibile to verify
immediately $M_\xi$ has determinant with the same $p$-adic valuation as $a_d^{m-d}$, because for
$s\leq{}k\leq{}r$ the matrix $B_k$ has determinant with valuation equal to $\sigma_k\ell_k(m-d) =
\left(v_p(a_{w_k})-v_p(a_{w_{k-1}})\right)(m-d)$.

Consequently $(\Theta_m^{(p)}:\Lambda_m^{(p)})$ is equal to the biggest power of $p$ dividing
$a_d^{m-d}$, for each $p$. We have proved:

\begin{corollary} \label{cor1}
The module $\Lambda_m$ has index equal to $|a_d|^{m-d}$ in $\Theta_m$.
\end{corollary}

\section{Proof of Theorem \ref{teo2}}
Before providing the proof of Theorem \ref{teo2}, let us recall a few facts of linear algebra. Let
$\scalar{\cdot}{\cdot}$ be the standard inner product on $\R^n$, and $u_1,\dots,u_k \in \R^n$,
for $k \leq n$. The Gram matrix of the $u_i$ is the $k\times k$ matrix defined as
\[
  G(u_1,\dots,u_k) = (\scalar{u_i}{u_j})_{1\leq i,j\leq k},
\]
and its determinant can be geometrically interpreted as the square of the volume of the
parallelepiped formed by the vectors $u_i$. If the $u_i$ can be completed with
$u_{k+1},\dots,u_n\in\R^n$ to a basis of $\R^n$, and $v_1,v_2,\dots,v_n \in \R^n$ are such that they
form a pair of biorthonormal bases, then we have
\[
  G(u_1,u_2,\dots,u_n) G(v_1,v_2,\dots,v_n) = I,
\]
(see \cite[\S 66.2, pag. 66-6]{hogben2007handbook}). Since $G(u_1,\dots,u_k)$ is a minor of
$G(u_1,u_2,\dots,u_n)$ and its complementary cofactor is $G(v_{k+1},\dots,v_{n})$, if the
$u_1,\dots,u_n$ form a parallelepiped of volume $1$ (and consequently $G(u_1,\dots,u_n)$ has
determinant $1$) we
have that
\[
  \det G(u_1,\dots,u_k) = \det G(v_{k+1},\dots,v_n).
\]
This follows from the properties of compound matrices (see \cite[Chap. 1, \S 4, pag.  21,
equation (33) in particular]{gantmacher1959theory}) related to what sometimes is also called
``Jacobi's Theorem'' \cite[\S 14.16]{gradstejn2000table}.

We will also need the following technical Lemma about the asymptotic behaviour of the determinant of
a perturbed Toeplitz matrix that arises as the Gram matrix of particular sets of vectors, see
Bump-Diaconis \cite{bump2002toeplitz}, Tracy-Widom \cite{tracy2002limit} and Lyons
\cite{lyons2003szego} for general results on this topic. Let $B(x)$ be a real
polynomial of degree $d$, and let $(B_i)_{1\leq{}i\leq\ell}$ be the row vectors of the matrix
$\mm{B}{\ell}$. Let $f_1,\dots,f_q$ be a finite set of vectors in $\R^d$, which will also be
considered as vectors in $\R^{d+\ell}$ turning all extra coordinates to $0$.

\begin{lemma} \label{lem1}
We have that
\[
  \det G(f_1,\dots,f_q,B_1,\dots,B_\ell) = \Oh(M^{2\ell})
\]
for each $M > M(B)$.
\end{lemma}

\begin{proof}
Let us start showing that
\[
\det G(B_1,\dots,B_\ell) \leq \Oh(M^{2\ell})
\]
for each $M > M(B)$. The determinant of banded Hermitiatian Toeplitz matrices can be easily
estimated via Trench's Formula \cite[Theorem~2.10, pag.~41]{bottcher2005spectral}. Let
$C(x)=\sum_{j=-r}^s{}c_jx^j$ be a Laurent polynomial, and let $g_n(z)$ be the row
\[
   g_n(z) = (1,z,z^2,\dots,z^{r-1},z^{n+r},z^{n+r+1},\dots,z^{n+r+s-1}).
\]
Let $\xi_1,\dots, \xi_k$ be the distinct roots of $C(x)$, and let $\mu_1,\dots,\mu_k$ be their
multiplicities. Define $G_n$ as the determinant of the $(r+s) \times (r+s)$ matrix $\Gamma_n$
whose first $\mu_1$ rows are $g_n(\xi_1),g_n'(\xi_1),\dots,g_n^{(\mu_1-1)}(\xi_1)$, whose next
$\mu_2$ rows are $g_n(\xi_2),g_n'(\xi_2),\dots,g_n^{(\mu_2-1)}(\xi_2)$, and so on.

Then $G_0 \neq 0$, and putting $D_{n-1}(C) = \det(c_{i-j})_{0\leq i,j \leq n-1}$ we have (by
Trench formula)
\[
   D_{n-1}(C) = (-1)^{ns} c_s^n \frac{G_n}{G_0},\qquad\text{for every }n\geq 1.
\]
Applying the formula with $C(x) = B(x)B(x^{-1})$ and $n = \ell$, we just have to show that
$c_s^nG_n=\Oh(M^n)$ for each $M > M(C) = M(B)^2$, for some $K(M)$. But the determinant of the
$(r+s)\times(r+s)$ matrix $\Gamma_n$ can be expanded as sum of monomials in the $\xi_i$ having
polynomials in $n$ as coefficients, where in each monomial $\xi_i$ appears with exponent smaller
than $(n+r+s)\mu_i$. Consequently estimating each monomial the determinant is
$\leq{}P(n)M(\frac{1}{c_s}C)^{n+r+s}$ for some polynomial $P(n)$ in $n$. Therefore
$c_s^nG_n=\Oh(M^n)$ for each $M > M(C)$, as required.

Now we have from Bump-Diaconis \cite{bump2002toeplitz} that a minor of a Toeplitz matrix obtained
deleting the first $r$ columns and a fixed set of rows, i.e. of the form
$D_{n-1}^\lambda(C)=\det(c_{\lambda_i-i+j})_{1\leq i,j \leq n}$ for some fixed partition
$\lambda=(\lambda_1,\lambda_2,\dots)$, is asymptotic for $n\rightarrow \infty$ to
$K\cdot{}D_{n-1}(C)$ for some constant $K$. On the other hand, let us expand the determinant of the
matrix
\[
  (G_{i,j})_{1\leq i,j \leq q+\ell} = G(f_1,\dots,f_q,B_1,\dots,B_\ell)
\]
along the first $p \geq d+q$ columns: the expression obtained is a sum of the form
\[ 
 \sum_{i_1 < \dots < i_p} \det(G_{i_u,v})_{1\leq u,v \leq p} \cdot C_{12\dots{}p}^{i_1i_2\dots{}i_p}(G),
\]
where the $\det(G_{{i_u},v})$ is non-zero only for a finite number of choices of the rows
$i_1,\dots,i_p$, while for $i_1,\dots,i_p$ fixed the cofactor $C_{12\dots{}p}^{i_1i_2\dots{}i_p}(G)$
is $\pm 1$ times a determinant of the form $D_{n-1}^\lambda(C)$, where $C(x) = B(x)B(x^{-1})$
and $\lambda$ is a partition depending only in the $i_1,\dots,i_p$. Consequently the
expansion is a sum of a fixed number of terms that are $\Oh(M^{2\ell})$ for each $M > M(B)$, and
hence is $\Oh(M^{2\ell})$ too.
\end{proof}

We require also the following Lemma, which provides for $m=\ell+d$ a basis of the $d$-dimentional
the lattice $\Lambda_m$ of vectors in $\Z^m$ killed by $\mm{A}{\ell}$ (i.e. form a linear recurrence
determined by $A(x)$) which has good properties with respect to the Gramian:

\begin{lemma} \label{lem2}
  There exist for each $m$ a basis $\omega_1,\dots,\omega_d$ of the lattice $\Lambda_m$ such that
  \[
   G(\omega_{r_1},\dots,\omega_{r_p}) \leq K \cdot M^{2m}
\]
for each subset of the $\omega_i$ and for each $M>M(A)$, for some constant $K$ not dependent of $m$.
\end{lemma}
\begin{proof}
  To construct the required basis, let $\zeta_1,\dots,\zeta_d \in \Q^m$ be such that the $j$-th
  coordinate of $\zeta_i$ is $\delta_{ij}$ for $1\leq i,j\leq d$, and define the remaining
  coordinates by the linear recurrence determined by $A(x)$. The $\zeta_i$ are clearly a basis of
  the $\Z$-module $\Theta_m$ of vectors in $\Z^d \times \Q^\ell$ which are killed by $\mm{A}{\ell}$,
  and we have by Corollary \ref{cor1} that $(\Theta_m : \Lambda_m) = |a_d|^\ell$.  Consequently, a
  basis $(\omega_i)$ of $\Lambda_m$ can be obtained by the $(\zeta_i)$ applying a matrix
  $W=(W_{i,j})_{1\leq i,j\leq d}$ with determinant $|a_d|^\ell$. Changing it by left multiplication
  by an element of $SL(d,\Z)$ the matrix $W$ can be taken in Hermite Normal Form \cite[\S
  23.2, pag. 23-6,23-7]{hogben2007handbook}, that is upper triangular and such that
  $|W_{i,j}| \leq |W_{j,j}|$ whenever $i<j$.

  We show that we can bound the Gramian of a subset of the $\omega_i$ in terms of the Gramians of
  all subsets of the $\zeta_i$ and the determinant of $(W_{i,j})$. Indeed, let
  $\omega_{r_1},\dots,\omega_{r_p}$, for $1 \leq p\leq d$ and $1\leq r_1 < \dots < r_p \leq d$, be a
  subset of the $\omega_i$. The quantity $\sqrt{G(\omega_{r_1},\dots,\omega_{r_p})}$ is the volume
  of the parallelepiped formed by the $\omega_{r_i}$ and is also equal to
\[
 \sup_{\phi \in \Phi_p} \phi(\omega_{r_1},\dots,\omega_{r_p}),
\]
the $\sup$ being taken within all the elements of the exterior power $\Lambda_p^\ast(\R^m)$ of the
form
\[
   \Phi_p = \left\{n_1\wedge n_2 \wedge \dots \wedge n_p,\text{ for orthonormal }
n_1,\dots,n_p\in\R^m\right\} \subset \Lambda_p^\ast(\R^m),
\]
where $\R^m$ is identified with the dual through $\scalar{\cdot}{\cdot}$. This can now be estimated
as
\begin{align*}
  &\ \sup_{\phi \in \Phi_p} \phi(\sum_{j=1}^d
W_{r_1,j}\zeta_j,\dots,\sum_{j=1}^d W_{r_p,j}\zeta_j) \\
  =&\ \sup_{\phi \in \Phi_p} \sum_{1\leq s_1,\dots,s_p \leq d} \phi (W_{r_1,s_1}\zeta_{s_1},\dots,
W_{r_p,s_p}\zeta_{s_p}) \\
  \leq&\  \sum_{1\leq s_1,\dots,s_p \leq d} \sup_{\phi \in \Phi_p} \ W_{r_1,s_1}\cdots W_{r_p,s_p}
 \cdot \phi (\zeta_{s_1},\dots, \zeta_{s_p})\\
  \leq&\  \sum_{1\leq s_1,\dots,s_p \leq d} |W_{s_1,s_1}\cdots W_{s_p,s_p}|\cdot \sup_{\phi \in
\Phi_p} \phi(\zeta_{s_1},\dots, \zeta_{s_p})\\
  \leq&\  \det(W_{i,j}) \cdot \sum_{1\leq s_1,\dots,s_p \leq d} \sup_{\phi \in \Phi_p}
\phi(\zeta_{s_1},\dots, \zeta_{s_p}),
\end{align*}
since we can discard the summands where $s_i = s_j$ for some $i \neq j$, and $(W_{i,j})$ is upper
triangular and with integral entries,
\begin{align*}
  \leq&\  \det(W_{i,j}) \cdot d^p \cdot \max_{1\leq s_1< \dots< s_p\leq d}
\sqrt{G(\zeta_{s_1},\dots,\zeta_{s_p})}.
\end{align*}

But let $B(x) = a_d^{-1}A(x)$ (so that $M(B) = |a_d^{-1}|\cdot M(A)$), and let
$(B_i)_{1\leq{}i\leq\ell}$ be the row vectors of the rational matrix $\mm{B}{\ell}$. The
  matrix $\mm{B}{\ell}$ can be completed to a square matrix with determinant $1$ inserting
the row vectors $e_1,\dots,e_d$ of the standard basis $(e_i)_{1\leq i \leq m}$ of $\R^m$, and the
$\zeta_i$ are dual to the $e_i$ in the basis $e_1,\dots,e_d,B_1,\dots,B_\ell$ with respect to the
standard scalar product $\scalar{\cdot}{\cdot}$. Consequently, if $\zeta_{r_1},\dots,\zeta_{r_p}$,
for $1\leq{}r_1<\dots<r_p\leq{}d$, are a subset of the $\zeta_i$ and $1\leq{}s_1<\dots<s_q\leq{}d$
is the complementary set of indices in $1,2,\dots,d$, the Gram determinant of the
$\zeta_{r_1},\dots,\zeta_{r_p}$ is the same as the Gram determinant of the vectors
$e_{s_1},\dots,e_{s_q},B_1,\dots,B_\ell$, which is $\Oh(N^{2\ell})$ for each $N > M(B)$ by Lemma
\ref{lem1}. Therefore the Gram determinant of the $\omega_{r_i}$ can be estimated with
$K|a_d|^{2\ell}N^{2\ell}$ for some constant $K$, and hence with $K M^{2\ell} = K'M^{2m}$ for each $M
> M(A)$.
\end{proof}

It is now possible to give the

\begin{proof}[Proof of Theorem \ref{teo2}]
As in the proof of Theorem \ref{teo1}, put $\ell=m-d$, and let $Q$ be defined as
\[
  Q = Q(A) = \Big\{ v \in \R^m\ :\ \mm{A}{\ell} \cdot v \in \Z^\ell \Big\}.
\]
Let $\omega_1,\dots,\omega_d \in \Z^m$ be any basis of the $d$-dimentional the lattice $\Lambda_m$
of vectors in $\Z^m$ which are killed by $\mm{A}{\ell}$ (i.e. form a linear recurrence determined by
$A(x)$), and let $e_1,\dots,e_m$ be the standard basis of $\R^m$.

Suppose by contradiction $\bar\epsilon < \epsilon < 1/M(A(x))$ and that $Q$ is
$\bar\epsilon$-dense, independently of $m$. Since $\epsilon > \bar\epsilon$ we have
$\pi_m(Q+[0,\epsilon]^m) = \T^m$, and since $Q$ is the union of integral translates of the
parallelepiped formed by $\omega_1,\dots,\omega_d$ the map $\pi_m : \R^m \rightarrow \T^m$ must be
surjective on the parallelepiped $\Pi$ formed by combinations with coefficients in $[0,1]$ of the
vectors $\epsilon e_1,\dots,\epsilon e_m,\omega_1,\dots,\omega_d$.

The map $\pi_m$ locally preserves the volume and the image of $\Pi$ is all $\T^m$, so the volume of
$\Pi$ must be $\geq 1$. But the volume of $\Pi$ can be computed as the sum of the volumes of the
parallelepipeds formed by all choices of $m$ vectors within
$\epsilon{}e_1,\dots,\epsilon{}e_m,\omega_1,\dots,\omega_d$. Note that the volume of the
parallelepiped formed by, say, $\epsilon e_{s_1},\dots,\epsilon
e_{s_q},\omega_{r_1},\dots,\omega_{r_p}$ with $p+q=m$ is not greater than
\[
  \epsilon^q \sqrt{G(\omega_{r_1},\omega_{r_2},\dots,\omega_{r_p})} \leq \epsilon^{m-d}
 \sqrt{G(\omega_{r_1},\omega_{r_2},\dots,\omega_{r_p})},
\]
being $q\geq m-d$ and $\epsilon \leq 1$. The total number of such parallelepipeds is
$\binom{m+d}{m}$, and taking a basis $\omega_i$ of $\Lambda_m$ via Lemma \ref{lem2} the volume of
$\Pi$ can be estimated as
\[
\mathrm{Vol}(\Pi) \leq \sqrt{K} \binom{m+d}{m}\epsilon ^{m-d} M^m.
\]
In particular it $\rightarrow 0$ as $m \rightarrow \infty$ if $M$ is such that
$M(A)<M<1/\epsilon$, as it is possible to choose since we assumed $\epsilon < 1/M(A)$.
\end{proof}

\bibliographystyle{plain}
\bibliography{biblio}

\begin{thebibliography}{10}

\bibitem{bottcher2005spectral}
A.~B{\"o}ttcher and S.M. Grudsky.
\newblock {\em {Spectral Properties of Banded Toeplitz Matrices}}.
\newblock Society for Industrial and Applied Mathematics Philadelphia, PA, USA,
  2005.

\bibitem{bump2002toeplitz}
D.~Bump and P.~Diaconis.
\newblock {Toeplitz minors}.
\newblock {\em Journal of Combinatorial Theory, Series A}, 97(2):252--271,
  2002.

\bibitem{cassels1957introduction}
J.W.S. Cassels.
\newblock {\em {An introduction to diophantine approximation, Cambridge Tracts
  45}}.
\newblock Cambridge University Press, New York, 1957.

\bibitem{cassels1971ign}
J.W.S. Cassels.
\newblock {\em {An introduction to the geometry of numbers. 2nd printing,
  corrected}}.
\newblock Berlin-Heidelberg-New York: Springer-Verlag, VII, 344 p, 1971.

\bibitem{everest1999heights}
G.~Everest and T.~Ward.
\newblock {\em {Heights of polynomials and entropy in algebraic dynamics}}.
\newblock Springer Verlag, 1999.

\bibitem{gantmacher1959theory}
F.R. Gantmacher.
\newblock {\em {The theory of matrices, Vol. 1}}.
\newblock Chelsea Pub. Co., New York, 1959.

\bibitem{gradstejn2000table}
I.S. Grad{\v{s}}tejn, I.M. Ry{\v{z}}ik, A.~Jeffrey, and D.~Zwillinger.
\newblock {\em {Table of integrals, series, and products}}.
\newblock Academic Press, 2000.

\bibitem{hogben2007handbook}
L.~Hogben, R.A. Brualdi, A.~Greenbaum, and R.~Mathias.
\newblock {\em {Handbook of linear algebra}}.
\newblock CRC Press, 2007.

\bibitem{lyons2003szego}
R.~Lyons.
\newblock {Szego limit theorems}.
\newblock {\em Geometric And Functional Analysis}, 13(3):574--590, 2003.

\bibitem{macdonald1995}
I.G. Macdonald.
\newblock {\em {Symmetric Functions and Hall Polynomials (2nd edition)}}.
\newblock Oxford University Press, New York, 1995.

\bibitem{tracy2002limit}
C.A. Tracy and H.~Widom.
\newblock {On the limit of some Toeplitz-like determinants}.
\newblock {\em SIAM Journal on Matrix Analysis and Applications},
  23(4):1194--1198, 2002.

\end{thebibliography}

\end{document}